\documentclass[a4paper,11pt]{article}
\usepackage[T1]{fontenc}
\usepackage[utf8]{inputenc}
\usepackage{csquotes}
\usepackage[british]{babel}
\usepackage{appendix}
\usepackage{amsfonts}
\usepackage{bbold}
\usepackage{mathtools}
\usepackage{amsmath}

\usepackage{relsize}
\usepackage{amssymb}
\usepackage{amsfonts}
\usepackage{dsfont}
\usepackage{amsthm}
\usepackage{enumitem}
\usepackage{thmtools}
\usepackage{stmaryrd}
\usepackage{authblk}
\usepackage{bm}
\newcommand{\vect}[1]{\bm{#1}}

\usepackage{listings}
\usepackage{xcolor}

\usepackage[a4paper,top=3cm,bottom=3cm,left=2.5cm,right=2.5cm,marginparwidth=1.75cm]{geometry}

\usepackage[dvipsnames]{xcolor} 
\usepackage{multirow} 

\usepackage{enumitem} 
\usepackage{array} 

\usepackage{hyperref}
\usepackage[capitalise]{cleveref}
\numberwithin{equation}{section}

\allowdisplaybreaks

\newtheorem{definition}{Definition}
\numberwithin{definition}{section}
\newtheorem{example}[definition]{Example}

\newtheorem{remark}[definition]{Remark}

\newtheorem{theorem}[definition]{Theorem}
\newtheorem{proposition}[definition]{Proposition}

\newtheorem{lemma}[definition]{Lemma}

\declaretheorem[sibling=definition]{Lemma} 

\newcommand{\R}{\mathbb{R}}

\newcommand{\N}{\mathbb{N}}
\newcommand{\E}{\mathbb{E}}
\newcommand{\D}{\mathbb{D}}
\newcommand{\p}{\mathcal{P}}

\renewcommand{\P}{\mathbb{P}}

\definecolor{rblue}{rgb}{0.0, 0.22, 0.66}
\definecolor{bred}{rgb}{0.8, 0.25, 0.33}
\definecolor{fgreen}{rgb}{0.24, 0.82, 0.44}
\usepackage{tcolorbox}
\tcbuselibrary{theorems}
\usepackage{epsfig}

\newtcbtheorem{mytheo}{Theorem}%
{colback=green!5,colframe=green!35!black,fonttitle=\bfseries}{th}
\newtcbtheorem{myprop}{Proposition}%
{colback=red!5,colframe=red!35!black,fonttitle=\bfseries}{th}
\newtcbtheorem{mydef}{Definition}%
{colback=blue!5,colframe=blue!35!black,fonttitle=\bfseries}{th}

\newtcbtheorem{myex}{Exemple}{fonttitle=\bfseries, colbacktitle=lime!20!white,coltitle=black!75!black, colback=lime!5, colframe=lime!45!black}{th}

\usepackage{enumitem}

\title{Moment duality and propagation of exchangeability}

\author[1]{Adrián González Casanova}
\author[2]{Ariel Offenstadt}
\author[3]{Arno Siri-Jégousse}

\affil[1]{School of Mathematics and Statistical Sciences and Biodesign institute, Arizona State University, Wexler Hall, 85281
Tempe, Arizona, USA }
\affil[2]{Universit\'e Paris Cit\'e, CNRS, MAP5, F-75006 Paris, France}
\affil[3]{IIMAS, UNAM, Mexico City, Mexico}

\date{}

\begin{document}
\maketitle
\begin{abstract}
We identify necessary and sufficient conditions for a class of random mappings to send exchangeable $\{0,1\}$-sequences to other exchangeable $\{0,1\}$-sequences. We call this property the propagation of exchangeability, and show that any mapping that propagates exchangeability induces a pair of forward- and backward-in-time processes that are moment duals. This establishes a transparent, tractable, and applicable connection between moment duality and the propagation of exchangeability. We illustrate the usefulness of our results by constructing lookdown models for $\Xi$-Fleming--Viot processes with frequency-dependent selection.
\end{abstract}

\section{Introduction}

Moment duality \cite{EK86,JK14,Lig85,SSV18} and exchangeability \cite{Ald82,Ald85,Hoo79,HS55,Kal92,Kin78,Kin78b} are two central concepts in probability theory, with deep applications in interacting particle systems, Bayesian statistics or population genetics.

Moment duality provides a powerful link between two stochastic processes: typically, a forward-in-time process describing type frequencies is related to a backward-in-time process describing genealogical quantities, such as the number of (potential) ancestral lineages. It relates the moments of a stochastic process taking values in $[0,1]$ to the probability generating function of a process taking values in $\mathbb{N}$. 
More precisely, two processes $\{W_t\}_{t}$ and $\{A_t\}_{t}$ are said to be moment dual if, for all $t \ge 0$, $x \in [0,1]$, and $n \in \mathbb{N}$,
\begin{equation}\label{e1}
    \mathbb{E}[W_t^n|W_0=x] = \mathbb{E}[x^{A_t}|A_0=n].
\end{equation}

Such dualities uniquely characterize the finite dimensional distribution of both processes and allow one to study complicated forward-in-time dynamics through simpler backward-in-time dynamics, or vice versa. Classical examples include the duality between the voter model and coalescing random walks, the duality between the Wright-Fisher diffusion and Kingman's coalescent, and its many extensions.

\begin{figure}[h!]
    \centering
    \includegraphics[width=0.5\linewidth]{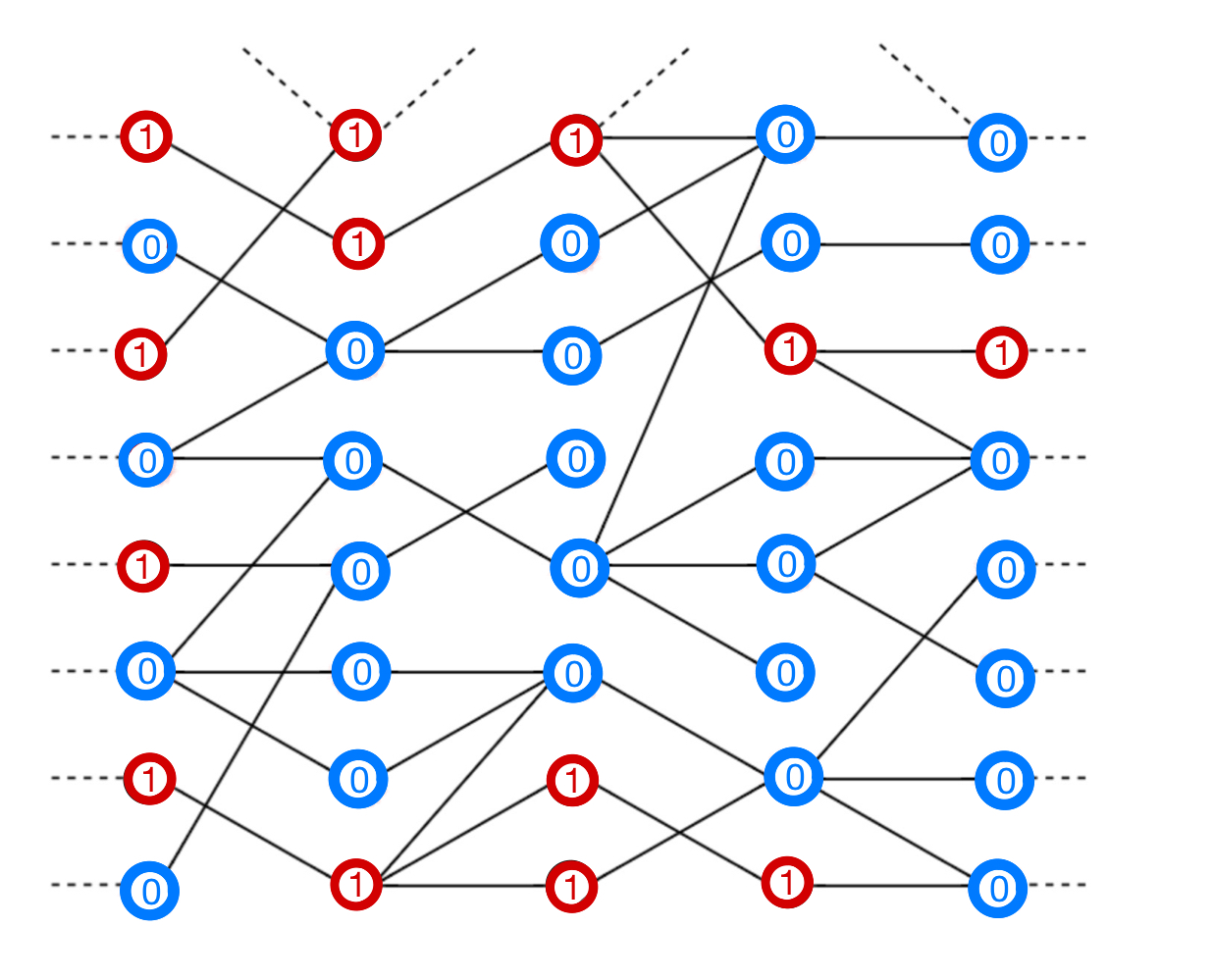}
    \caption{This figure depicts a graph with vertex set $\mathbb{Z}\times\mathbb{N}$. A vertex $(k,i)$ represents {individual $i$ in generation $k$}. Edges connect individuals in generation $k$ to individuals in generation $k+1$ and represent potential ancestral relationships. Each individual inherits its type from its potential ancestors in the previous generation according to the following rule: an individual $(k,i)$ is of type~$1$ if and only if all of its potential ancestors are of type~$1$; otherwise, it is of type~$0$. A sequence of zeros and ones is transformed into a new sequence. In such a model, when does exchangeability propagate? 
} 
    \label{Fig1}
\end{figure}

On the other hand, a sequence of random variables is exchangeable if its joint distribution is invariant under finite permutations. Formally, a sequence of random variables is exchangeable if in distribution $\{X{(i)}\}_{i}=\{X{(\sigma(i))}\}_{i}$ for any finite permutation $\sigma:\mathbb{N}\mapsto \mathbb{N}$.
Exchangeability is a fundamental property of random sequences and has been extensively studied since the work of de Finetti.  While exchangeable sequences are generally not independent, their dependence structure is transparent. De Finetti's theorem states that any $\{0,1\}$-valued exchangeable sequence can be represented as a sequence of conditionally independent Bernoulli random variables given a random parameter $\Theta$. Equivalently, the probability of observing $n$ ones in a sample of size $n$ is given by 
\begin{equation}\label{e2}
    \mathbb{E}[\prod_{i=1}^nX(i)|\Theta]=\Theta^n.
\end{equation}

Comparing equations \eqref{e1} and \eqref{e2} reveals a similarity: both frameworks encode probabilities of sampling configurations through powers of a random variable. This observation suggests a deep connection between exchangeability and moment duality, beyond their frequent co-occurrence in classical models.

A canonical example illustrating co-occurrence is provided by Cannings models. In these fixed-size population models with discrete generations, the vector giving the number of children of all individuals in a generation is supposed to be exchangeable. As a result, if the vector of types is initially exchangeable, then it remains exchangeable in every subsequent generation. Moreover, it is well known that the type-frequency process in Cannings models is moment dual to the block-counting process describing the number of ancestral lineages backward in time. However, recent works \cite{GPSJ,KJJS, SJW} suggest that  the exchangeability assumption on the progeny vector can be relaxed without losing moment duality.

The goal of this paper is to make the relationship between exchangeability and moment duality explicit and structural. Rather than starting from a specific model, we address the following general question: given an infinite exchangeable sequence of random variables and a random mechanism that transforms this sequence into a new one, under which conditions is exchangeability preserved? In other words, which random functions from the space of infinite sequences of zeros and ones to itself send exchangeable sequences to exchangeable sequences?

To answer this question, we introduce a framework based on random mappings that assign to each integer a set of integers. This mapping will be interpreted as assigning to any individual a set of potential ancestors in the previous generation (see Figure \ref{Fig1}). 
The key observation is that exchangeability of the resulting type sequence does not require the ancestral assignments themselves to be exchangeable. Instead, we identify a precise and non-trivial condition, that we call the label-forgetfulness property, which is necessary and sufficient for the propagation of exchangeability. Informally, label-forgetfulness means that the distribution of the number of potential ancestors of a sample depends only on the sample size and not on the specific labels involved.

Our first result, Theorem~\ref{forgetful}, shows that this label-forgetfulness property is equivalent to propagation of exchangeability in the $\{0,1\}$-valued case. 
Building on this, we show in Theorem~\ref{thduality} that whenever exchangeability propagates, the forward-in-time de Finetti frequency process is in moment duality with the backward-in-time process counting ancestral family sizes. 
In this sense, moment duality and propagation of exchangeability are two faces of the same coin.

Beyond its conceptual clarity, this framework has concrete applications. It provides a simple and transparent method for proving exchangeability. As a proof of concept, we focus on population genetics models, where the label-forgetful machinery leads to transparent constructions of lookdown processes that extend the work of Bah and Pardoux \cite{BP15}, see Theorems \ref{lazy} and \ref{theomix}. 
Indeed, using our approach, we obtain finite representations of stochastic differential equations, such as the $\Xi$ -Wright-Fisher jump-diffusion with frequency-dependent selection, and recover their dual genealogical processes in a unified manner. \\

\section{Conditions for propagation of exchangeability}
\label{partI}
\subsection{The label-forgetfulness property}

     Consider two independent sequences \(\{X_0(i)\}_i\) and \(\{L(i)\}_{i}\), where for each $i\in\N$, $X_0{(i)}$ is a \(\{0,1\}\)-valued random variable  and 
     \(L(i)\) is a \(\mathcal P(\mathbb{N})\)-valued random variable, where \(\mathcal P(\mathbb{N})\) is the space of non-empty subsets of \(\mathbb{N}\).
     Define the sequence \(\{X_1{(i)}\}_{i  }\)
     such that 
     \begin{equation}
     \label{eq:transmission}
         X_1{(i)}=\inf_{j \in L(i)} X_0{(j)}.
     \end{equation}

\begin{figure}
    \centering
    \includegraphics[width=0.5\linewidth]{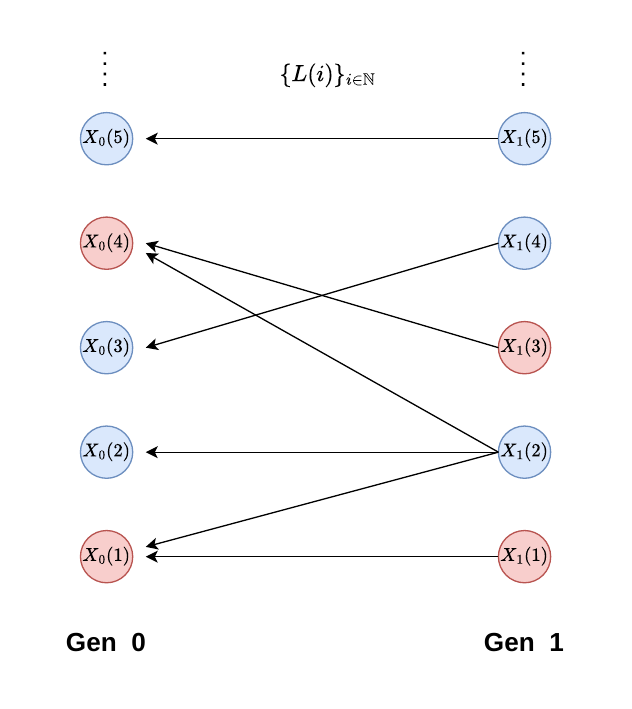}
    \caption{Illustration of the propagation of types from an initial generation $0$ to the next generation $1$, induced by a realization of the random mapping $L$ and formula (\ref{eq:transmission}). The color blue represents type $0$ while the color red represents type $1$. The former is the {strong} type: if there exists $j\in L(i)$ such that $X_0(j)=0$, then $X_1(i)=0$. This is the case in particular for $X_1(2)$. }
    \label{fig:rougeetbleu}
\end{figure}
This transmission of the {types} $0$ and $1$ is illustrated in Figure \ref{fig:rougeetbleu}.
We say that the random mapping $L$ from \(\mathbb{N}\) to \(\mathcal P(\mathbb{N})\)  {propagates exchangeability} if whenever $\{X_0{(i)}\}_{i}$ is exchangeable, then $\{X_1{(i)}\}_{i }$ is also exchangeable.
By extension, one can also apply this terminology directly to the sequence $\{L(i)\}_{i}$. It is clear that if $\{L(i)\}_{i}$ is exchangeable, then $\{L(i)\}_{i}$ propagates exchangeability, but the converse is not true. 
The trivial set $\{\{1\},\{2\},...\}$ obtained when $L(i)=\{i\}\text{ for any } i\in \N $ is obviously not exchangeable although it propagates exchangeability.
  It is important to understand that exchangeability is involved at two different levels, that of the sequences $X_0$ and $X_1$, and that of the mapping $L$, which shapes the links between $X_1$ and $X_0$. 

We want to establish a simple property ensuring the propagation of exchangeability.
First, we naturally extend $L$ to a mapping of $\p(\N)$ to itself by denoting,
for each \(S \subset \mathbb{N}\), \(L(S) = \bigcup_{i \in S} L(i)\). 
We may still use the notation $L(i)=L(\{i\})$ without ambiguity. 

We say that $L$ is {label-forgetful} if  \(|L(S)|\stackrel{d}{=}|L(S')|\) for all samples $S,S'\subset \N$ such that \(|S| = |S'|\).
This property is crucial in our framework.

\begin{theorem}
   The mapping \( L \) propagates exchangeability if and only if it is label-forgetful.
\end{theorem}\label{forgetful}
\begin{proof}
Fix $m\in\N$. Let $n_1,\dots,n_m\in\N$, $v_1,\dots,v_m\in \{0,1\}$ and assume that $L$ is label-forgetful. 
Denote by  $Z$ the set of indices $j$ such that  $v_j = 1$, and $Z^c$ its complement in $\{1,...,m\}$.
  Then,
\begin{center}
\begin{math}
\begin{aligned}
    \mathbb{P}\big(X_{1}{(n_1)}=v_1,\dots,X_{1}({n_m})=v_m\big) 
    & = \mathbb{E}\big[\prod_{ j=1}^ m\mathbb{1}_{\{X_{1}({n_j})=v_{j}\}} \big] \\
    & = \mathbb{E}\big[\prod\limits_{j\in Z} X_{1}({n_j})\prod\limits_{j\in Z^c}(1-X_{1}({n_j}))\big] \\
    & = \displaystyle\sum_{G \subset Z^c} (-1)^{|G|} \mathbb{E}\big[\prod\limits_{j\in Z}X_{1}({n_j})\prod\limits_{j\in G} X_{1}({n_j}) ] \\
    & = \displaystyle \sum_{G \subset Z^c} (-1)^{|G|} \mathbb{P}(\underset{j \in Z\cup G}{\inf}X_{1}({n_j})=1 )\\
    & = \displaystyle \sum_{G \subset Z^c} (-1)^{|G|} \mathbb{P}(\underset{i\in L({Z\cup G})}{\inf}X_{0}({i})=1  )\\
    & = \displaystyle \sum_{G \subset Z^c} (-1)^{|G|} \mathbb{P}(\underset{i\in L\big({\big[|Z\cup G|\big]}\big)}{\inf}X_{0}({i})=1 ).
\end{aligned}
\end{math}
\end{center}
If $\{X_0{(i)}\}_{i}$ is exchangeable, the last quantity depends only on $m$ and $(v_1,\dots,v_m)$ and thus $(X_{1}({n_1}),\dots,X_{1}({n_m}))$ is exchangeable. Since $m$ and $n_1,\dots,n_m$ are arbitrary, the sequence $\{X_1{(i)}\}_{i}$ is also exchangeable.

Reciprocally, assume that there are two subsets $S, S'$  of the same size and such that $|L(S)|$ and $|L(S')|$
differ in distribution. 
Then $\mathbb{P}(X_{1}({i})=1,\text{for all } {i \in S}) $ and $\mathbb{P}(X_{1}({i})=1,\text{for all } {i \in S'}) $  also differ, which prevents the exchangeability of the sequence $\{X_{0}({i})\}_{i} $.
\end{proof}

Label-forgetfulness is therefore a necessary and sufficient condition that is often simple to check. In particular, this is the case for the following basic examples.

\begin{example}[Branching]\label{ex:branching}
Let $\{\xi_i\}_{i}$ be a sequence of i.i.d. $\N$-valued random variables. 
Define $\{L(i)\}_{i}$ by the rule: $j\in L(i)$ if and only if $\sum_{l=1}^{i-1}\xi_l<j\leq \sum_{l=1}^i \xi_l$. 
\end{example}
This example is related to the Skeleton Chain of the dual of the Wright-Fisher diffusion with selection in a random environment \cite{BCM}, which was introduced in \cite{GSW25}.

\begin{example}[Discrete lookdown]\label{ex:coal}
The principle is to draw a group of individuals which all select the smallest element of this group, while the remaining select the smallest {available} element in ascending order. 
Rigorously, let $\{B_i\}_{i}$ be a sequence of independent Bernoulli random variables with parameter $y\in(0,1)$. Let $R=\inf\{i:B_i=1\}$. 
Each level sequentially picks a unique integer, starting from level 1, with the following rule:
\[
L(i) = \begin{cases} R & \text{if } B_i = 1 \\ i+\mathbb{1}_{i<R}-\sum_{l=1}^iB_l & \text{if } B_i = 0 \end{cases}
\]

\end{example} 
This example can be thought of as the skeleton chain of the classical modified lookdown construction by Donnelly and Kurtz \cite{DK99} and its multiple merger version \cite{Seven05} for the case $\Lambda = \delta_y$.

\begin{example}[Mixture of label-forgetful mappings]\label{ex:mix}
Consider two label-forgetful mappings ${L^1}$ and ${L^2}$ and an independent Bernoulli random variable $Y$.
Then, the mapping $Y{L^1}+(1-Y) L^2$ is label-forgetful.
\end{example}

For applications in population genetics, we  refer to $X_k{(i)}$ as the {type} of the $i$-th individual in generation $k$. 
On the other hand, the random mapping $L$ assigns to each individual in a given generation a list of {potential ancestors} one generation earlier.
Label-forgetfulness ensures that only the size of a sample affects the number of their potential ancestors, not the particular labels that compose it. As mentioned in the proof of Theorem \ref{forgetful}, if some individuals could bias the size of their ancestry, this would naturally prevent the exchangeability of the types. In the next section, we build and study a whole population process based on a label-forgetful mapping.

\subsection{The de Finetti process and duality}


We assume throughout this section that $L$ propagates exchangeability. Successive and independent draws of $L$-distributed random mappings allow us to build a genealogical graph  connecting each generation to the preceding one by associating each individual (and thus each sample) to a list of potential ancestors. Once this graph is fixed, we let the types spread from the initial generation. It is crucial in this regard that the construction of the graph is independent of the assignment of types, which is done at generation $0$.   

 More precisely, assume that \( \{X_0{(i)}\}_{i} \) is an exchangeable sequence with values in \(\{0,1\} \) such that $\P(X_0{(i)}=1)=x\in(0,1)$ for any $i\in\N$. Given \(\{X_{k-1}(i)\}_{i}\), define \(\{X_{k}(i)\}_{i}\) by the rule 
 \begin{equation}
\label{X_forward} 
X_k{(i)} = \inf_{j \in L_{k}(i)} X_{k-1}(j)
 \end{equation} where \(\{L_{k}\}_{k}\) are i.i.d. copies of $L$, and $L_k$ is independent of \(\{X_{k-1}(i)\}_{i}\). 
 We are now in position to study two processes obtained by reading the resulting graph in both directions, a forward-in-time process monitoring the types and a backward-in-time process following the size of ancestral families. An illustration of the process is given in Figure~\ref{fig:fullpro}.

\begin{figure}
    \centering
    \includegraphics[width=0.9\linewidth]{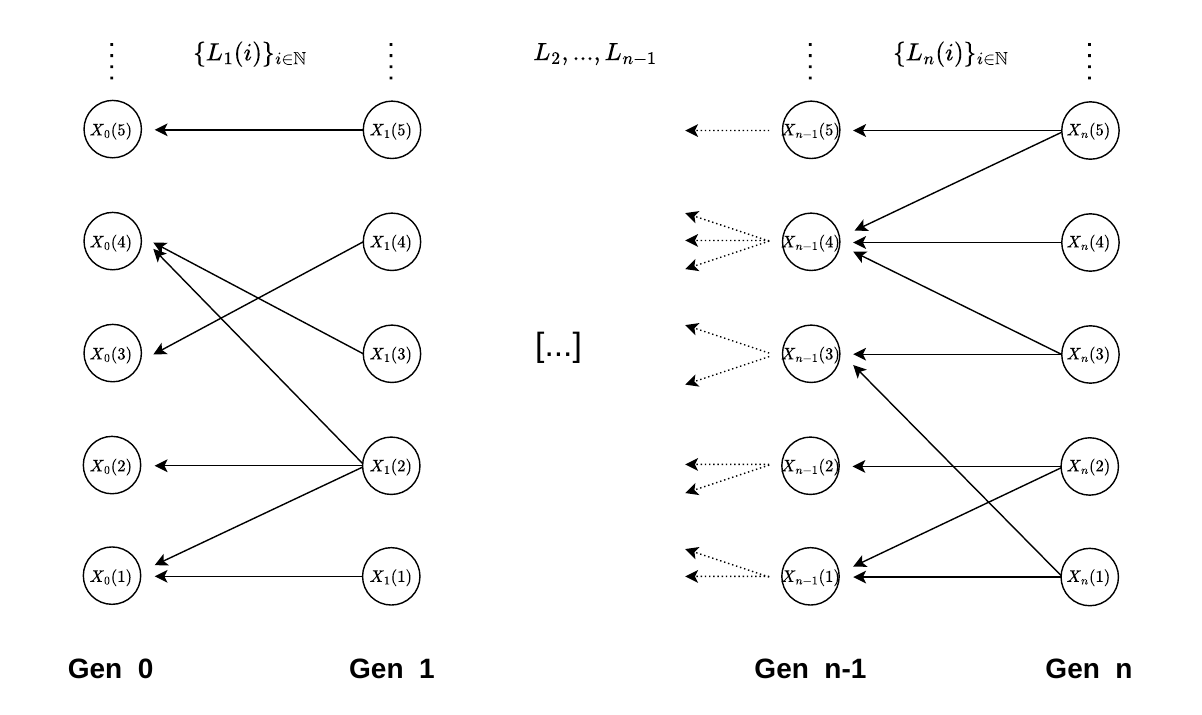
    }
    \caption{Illustration of the graph corresponding to the genealogical tree induced by the realization of i.i.d. random mappings $\{L_k\}_k$. The construction of the graph is in particular independent of the assignment of types at generation $0$, types that later spread through the graph based on formula ($\ref{X_forward}$).}
    \label{fig:fullpro}
\end{figure}

The propagation of exchangeability of $L$ implies that for every $k \in \N$, the sequence $\{X_{k}(i)\}_{i}$ is exchangeable. 
Through de Finetti's theorem, there exists for every $k$ a random variable  $\Theta_k := \lim_{p \to \infty} \frac{1}{p} \sum_{i = 1}^pX_k{(i)}$ corresponding to the empirical frequency of type $1$ at generation $k$. 
Moreover, conditionally on $\Theta_k$, the $X_k{(i)}$'s are independent and distributed according to a Bernoulli law of parameter $\Theta_k$. We have, for any $i,k\in\N$,  $$ \E[X_k{(i)}|\Theta_k]=\P(X_k{(i)}=1|\Theta_k)=\Theta_k.$$

The construction also implies that $\{\Theta_k\}_{k}$ is a $[0,1]$-valued Markov chain. Indeed, let $k,n\in\N$ and denote $[n]=\{1,\dots,n\}$. We have that 

   \[
\begin{aligned}
\mathbb{E}\big[ (\Theta_{k+1})^n \big] & = \mathbb{E}\bigg[ \mathbb{E}\big[ \prod_{i=1}^{n}X_{k+1}(i) \big| \Theta_{k+1} \big] \bigg] \\ & = \mathbb{E}\bigg[ \prod_{i=1}^{n}X_{k+1}(i)  \bigg] \\
& = \mathbb{E}\bigg[ \prod_{i=1}^{|L_{k+1}([n])|} X_k(i)\bigg] \\ 
& = \mathbb{E}\bigg[ (\Theta_k)^{|L_{k+1}([n])|}\bigg].
\end{aligned}
\]
 As $\Theta_k$ is $[0,1]$-valued, its law is fully determined by its moments. 
 Therefore, since $\Theta_k$ and $L_{k+1}$ are independent, the conditional distribution of $\Theta_{k+1}$ with respect to $\mathcal{F}_k:=\sigma(\Theta_l,0\leq l\leq k)$ depends only on $\Theta_k$.
 We call the Markov process $\{\Theta_k\}_k$ the \textit{de Finetti process}.

{We now turn to the backward process induced by our framework.}
Let us define the $\N$-valued Markov chain $\{A_k\}_{k}$  with the following transitions. For any $n,m\in\N$, 
\begin{equation}
\label{block_counting}
\P(A_1=m|A_0=n)=\P(|L([n])| = m).    
\end{equation} 
It is called the {block-counting} process of the backward-in-time genealogical process induced by $L$. 
Indeed, if $A_0=n$, then $A_k$ is equal in distribution to the number of potential ancestors of a sample of size $n$, $k$ generations in the past, namely after the composition of $k$ independent $L$-distributed random variables. 
Through label-forgetfulness, we identify the law of $A_k$ started at $n$ with that of $|L_{1}\circ\cdots\circ L_{k}([n])|$.

\begin{theorem}[Moment duality]\label{thduality}
   Assuming L is label-forgetful, the forward and backward processes are moment dual: \begin{equation}
    \label{duality}
\mathbb{E}_x\big[\Theta_k^n\big]=\mathbb{E}_n\big[x^{A_k}\big]  
   \end{equation}
   where $\mathbb{E}_x$ stands for the expectation associated to the process $\{\Theta_k\}_k$ started at $x$, and $\mathbb{E}_n$ stands for the expectation associated to the process $\{A_k\}_k$ started at $n$.
\end{theorem}
\begin{proof}
   The proof technique is known as sampling duality, and is based on studying the relationship between the types of a sample of individuals in generation $k$, which is related to $\Theta_k$, and their  $A_k$ ancestors in generation $0$. More precisely, we look at the probability that the first $n$ individuals in generation $k$ (or equivalently any sample of size $n$ in generation $k$) all have the type $1$. This is given by
\[
    \P_{x}(X_k(i)=1, i\in[n])   = \mathbb{E}_x\big[\Theta_k^n\big] \\
\]
where we use the conditional independence of the $X_k(i)$'s on $\Theta_k$. 

Now, summing over the number of potential ancestors at generation $0$ of these first $n$ individuals, we obtain  
\[
\begin{aligned}
    \P_{x}(X_k(i)=1, i\in[n]) 
    & = \displaystyle \sum_{m\geq1}\mathbb{P}_n(A_k=m)\mathbb{P}_x(X_0(j)=1, j\in[m]) \\
    & = \displaystyle \sum_{m\geq1}\mathbb{P}_n(A_k=m)x^m \\
    & = \mathbb{E}_n\big[x^{A_k} \big].
\end{aligned}
\]
\end{proof}

 Formula \eqref{duality} on the moments of $\Theta_k$ for any $k$ characterizes the semi-group of the process $\{\Theta_k\}_k$. In some cases, the transitions of one of the two processes are simpler to study and allow us to find those of the dual. 

\begin{remark}
One could also use this methodology with time-inhomogeneous processes. As long as the mappings $L_1,L_2,\dots$ are independent and label-forgetful, the results remain correct.
\end{remark}

Moment duality relationships are often derived in continuous time, using the generator method. The next result shows how to bridge the gap from our discrete construction to those continuous-time processes.
The aim is to show that the de Finetti process $\{\Theta_k\}_k$ with a suitable time scaling can approximate classical forward-in-time processes.

\begin{proposition}[Approximation of forward processes]
\label{theo_cv}
Let $\{A^N\}_{N}$ be a sequence of block-counting processes associated with label-forgetful mappings $\{L^N\}_{N}$. 
For each $N\geq1$, define $\Theta^N$ as the de Finetti process dual to $A^N$ from Theorem \ref{thduality}.  Assume that there exists a continuous-time Markov process $\{A_t\}_{t}$ taking values in $\N$ such that 
\begin{enumerate}
    \item[i)] $\{A_t\}_{t}$ is moment dual to a $[0,1]$-valued Feller Markov process $\{W_t\}_{t}$. 
    \item[ii)] There exists a diverging sequence $\{a_N\}_N$ such that for any fixed $t\in[0,T]$,
    $A^N_{ \lfloor a_Nt\rfloor }$ converges in distribution towards $A_t$
    whenever $A^N_0\overset{a.s.}{=}A_0$.
\end{enumerate}
Then, defining $ W^N_t=\Theta^N_{\lfloor a_N t\rfloor}$, the sequence $\{W^N\}_{N}$ weakly converges towards $W$ in $\D([0,T],[0,1])$.
\end{proposition}

\begin{proof}

Let us start considering the convergence of semigroups.
Write $C_0$ for the space of continuous bounded functions on $[0,1]$ equipped with the uniform norm, and $\mathcal{L}(C_0)$ for the set of linear operators from $C_0$ to itself. 
Let $\{T_t\}_t$ be the semigroup of $W$, and define, for $N\in\N$ and $t\in[0,T]$, $T_{N,t}\in \mathcal{L}(C_0)$ such that for any $f\in C_0$, $T_{N,t}f(x)=\mathbb{E}_x[f({W}^N_{ t}$)]. 
Following Theorems 17.25 and 17.28 in \cite{Kal02}, the targeted convergence results from the existence of the limiting process, which is assumed, and the strong convergence of $T_{N,t}$ to $T_t$ in $\mathcal{L}(C_0)$ for every $t\in[0,T]$, which we will show, and is equivalent to the convergence in $C_0$ of $T_{N,t}f$ to $T_tf$ for any $f$. 
Plainly, it writes
\begin{equation}
\label{onveut}
    \underset{x\in[0,1]}{\sup} \big|T_{N,t}f(x)-T_tf(x) \big| \underset{N\to\infty}{\to} 0.
\end{equation}
Actually, since all functions of $C_0$ can be uniformly approximated by polynomials, it is enough to show the uniform convergence for functions of the form $f_p(x)=x^p$.

{Now we want to pass from simple to uniform convergence}.
Let $x\in[0,1]$, $t\in[0,T]$ and $p\in\N$. Using the duality results induced by Theorem \ref{thduality} and assumption $i)$, as well as the convergence in $ii)$, we have that 
\begin{equation}
\label{conv}
\mathbb{E}_x[f_p({W}^N_{ t})]=\mathbb{E}_x[f_p({\Theta}^N_{ {\lfloor a_N t\rfloor}})]=\mathbb{E}_p[x^{A^N_{\lfloor a_N t\rfloor}}]\underset{N\to\infty}{\to}\mathbb{E}_p[x^{A_t}]=\mathbb{E}_x[W_t^p].
\end{equation}

We can view $\eqref{conv}$ as a pointwise convergence of $T_{N,t}f_p$ to $T_tf_p$ on $[0,1]$. Furthermore, since $T_{N,t}f_p$ is increasing on the compact $[0,1]$, and $T_tf_p$ is continuous as the probability generating function of $A_t$, this pointwise convergence also implies uniform convergence on $[0,1]$ through Dini's theorem and \eqref{onveut} follows. 
\end{proof}

\section{Discrete lookdown constructions}
\label{Xi}

The fundamental duality result in population genetics is the moment duality established between the Wright-Fisher diffusion and Kingman's coalescent block-counting process \cite{EK86}, providing a first example of a relation between forward- and backward-in-time processes in random evolutionary models. 
This duality relation was later generalized at a measure-valued ({resp.} partition-valued) level by \cite{Per92}, linking the Fleming-Viot process with Kingman's coalescent.

In \cite{DK96}, a countable representation of the Fleming-Viot process and the Wright-Fisher diffusion was obtained, giving an intuitive and graphical construction of those limit processes. Roughly speaking, particles of this infinite system reproduce after an exponential time and their child kills and replaces an already existing particle in the infinite system. A smart relabeling framework allows to consider infinitely many individuals at once. 
  This countable representation, known as the lookdown construction, provides a graphical framework where both the diffusion and the coalescent can be built on.

 The lookdown construction was then modified in \cite{DK99}, providing yet another technical breakthrough. 
    In this version, particles are labeled by a level that informs about their time to remain in the system. Indeed, they can be pushed upwards when particles  below them reproduce.
    The construction is obtained through an exchangeability argument, which seems counterintuitive at first sight, since accounting for the levels may suggest that particle types are not exchangeable.
     The modified lookdown construction is important because it establishes a ranking between particles at any time.
     Among many applications, it is possible to study fixation lines \cite{H15,PW}, which track genealogical relations starting from the common ancestor.
Other  applications such as evolving coalescents \cite{PW} or the evolving size of ancestral families \cite{DDS} were later developed.
This modified lookdown construction was generalized to a wider class of Fleming-Viot processes and coalescents in \cite{Seven05,Five09}. 
Particle representations were also developed for more sophisticated models \cite{EK, KR} considering
evolving real-valued levels or variable size populations.

{In \cite{Seven05, Five09, DK99}, the models under consideration are neutral in the sense of evolutionary biology: no type carries a selective advantage over the others; yet selection is expected to play a central role in realistic modeling. 
The first duality result incorporating selection is established between the Ancestral Selection Graph (ASG), a branching-coalescing process, and the Wright-Fisher diffusion with selection  \cite{KN97}. 
However, allowing types to have different fitness and priorities significantly complexifies the exchangeability arguments lying in lookdown-like constructions, since the levels can no longer be assigned independently of the type carried by each particle. 
A construction accommodating the Wright-Fisher diffusion with weak selection was nonetheless achieved in \cite{BP15}.}

In this section, we use the machinery developed in section \ref{partI} to shed a new light on how the exchangeability can be preserved in lookdown-like constructions, even in cases where reproduction events seem to highly depend on the levels of particles. To do so, we build a discrete lookdown model, which we believe is more transparent than its continuous counterpart because the exchangeability arises, at a discrete-time level, as a consequence of the propagation of exchangeability. 
Aiming also at incorporating various selection mechanisms, we provide in Theorem~\ref{lazy} a Poissonian construction of a {lookdown graph} on which can be read both a backward-in-time branching-coalescing process, and its moment dual forward-in-time jump-diffusion in a general case involving simultaneous multiple mergers, binary mergers and branching events. 
This allows us to create de Finetti processes approximating the Wright-Fisher diffusion with jumps, weak selection and frequency-dependent selection, thus extending the work of \cite{BP15}.

For a mapping $L$, define $L^{-1}(i)=\{j,i\in L(j)\}$ for any $i$. 
We say that $L$ satisfies a lookdown principle if for all $i$, $L(i)\neq\emptyset$, $L^{-1}(i)\neq\emptyset$ and whenever $i_1<i_2$, then $\min L^{-1}(i_1)\leq \min L^{-1}(i_2)$.

\begin{theorem}[The lookdown probability space]
\label{lazy}
Set $\Delta = \{\vect{s}=(s_1,s_2,...):s_1\geq s_2\geq...\geq0,\displaystyle\sum s_l\leq1 \}$  for the infinite-dimensional simplex and $C=\{\vect{i}=(i_1,i_2)\in\mathbb{N}^2,i_1<i_2\}$ for the set of ordered pairs of integers. Denote also by $\vect{p}:=\{p_j\}_{j}$ a probability measure on $\N$ with finite expectation and probability generating function $\Phi$.
There exists a Poisson point process
\[
\Pi \subset \mathbb{R}_+ \times \Delta\times [0,1]^{\N}\times C\times\N^2,
\]
  such that one can construct the following using $\Pi$.

\begin{enumerate}
\item[i)] A family of random mappings $L_t$ that satisfy a lookdown principle for any $t$
and a $\N$-valued branching--coalescing process $\{{A}_t\}_{t}$ with transition rates 
\begin{equation}\label{eq:dynA}
\begin{cases} n\to n+j & \text{at rate } np_j \\n\to n-j & \text{at rate } c\binom{n}{2}\mathbb{1}_{j=1}+q_{n,n-j} \\
\end{cases}
\end{equation}
where $j\in \mathbb{N}$ and the array $\{q_{n,m}\}_{n,m}$ for $m=1,2,...,n-1$ provides the transition rates of the block-counting process of a $\Xi$-coalescent.
\item[ii)] A $[0,1]$-valued frequency process $\{W_t\}_{t}$,
   strong solution to the stochastic differential equation
        \begin{align}\label{eq:EDSW}
           dW_t & = \sqrt{W_t(1-W_t)}dB_t + f(W_t)W_t(1-W_t)dt \notag \\ & \qquad + \int_{\Delta\times[0,1]^{\N}}  \sum_{l\ge1} s_l \left(\mathbb{1}_{\{u_l \leq W_{t^-}\}} - W_{t^-}\right) M(dt, d\vect{s}, d\vect{u}), 
        \end{align}
        where $M$ is the restriction of $\Pi$ to $\R_+\times\Delta\times[0,1]^{\N}$ and $f(x):=c(1-\Phi(x))/(1-x)$ for some constant $c>0$.  
\end{enumerate}
Moreover, the processes $\{W_t\}_{t }$ and $\{{A}_t\}_{t}$ are moment dual {and can be read as forward and backward in time processes on the same graphical construction}. 

\end{theorem}

The transition rates of the block-counting process of the $\Xi$-coalescent can be found in  \cite{MG21}, Proposition 2.1. 
The atoms of the Poisson point process $\Pi$ defined in Theorem~\ref{lazy} correspond to different kinds of events for the backward and forward processes. 
More precisely, if $(t,\vect{s},\vect{u},\vect{i}, \vect{j})$ is a point in $\mathbb{R}_+ \times \Delta\times [0,1]^{\N}\times C\times\N^2$, then 
the sequences $\vect{s}$ and $\vect{u}$ encode simultaneous and multiple coalescence events in the backward-in-time process, while the pair vector $\vect{i}$ encodes binary coalescences, and the pair vector $\vect{j}$ encodes branching events.
For the forward-in-time process defined in \eqref{eq:EDSW}, $\vect{s}$ and $\vect{u}$ encode the jump term, while $\vect{i}$ encodes the diffusive term, and  $\vect{j}$ the drift.

 The proof of Theorem \ref{lazy} is developed in the following three sections. 
 First, we detail in section \ref{xi_event} how an element of  $\Delta\times[0,1]^\mathbb{N}$ can be transformed into an exchangeable partition inducing a classical simultaneous and multiple merger event. 
In section \ref{discrete}, we introduce the intensity measure of $\Pi$ and we use a discretization of the time and atoms of $\Pi$ restricted to $\mathbb{R}_+\times\Delta\times[0,1]^\mathbb{N}$ to build forward and backward processes converging towards the adequate continuous-time process. 
In section \ref{sec:Xi+K}, we incorporate the missing parts of the Poisson point process to recover the general case defined in Theorem~\ref{lazy}. 

The framework of the proof can be simplified thanks to an alternative technique to handle a mixture of various events, focused on label-forgetful mappings rather than atoms of a Poisson random measure.
This is developed in section~\ref{sec:mixture}.

\color{black}
\subsection{A  mapping associated to the lookdown construction}
\label{xi_event}

Recall the {infinite simplex} $\Delta$ with its origin $0=(0,\dots)$ and set for any $\vect{s}\in\Delta$, $s_0=1-\sum s_l$. 
Also consider the subset
$$\Delta^* = \{\vect{s}\in\Delta:\displaystyle\sum s_l=1 \}. $$

In the following, we detail how each element $(\vect{s},\vect{u})\in\Delta\times[0,1]^{\N}$  can be transformed into a partition in $\p_{\infty}$ and how it induces a mapping inspired by the so-called {paintbox procedure} \cite{Kin78,Ber06}.

Assume first that $\vect{s}\in\Delta^*$, and writing $\vect{u}=\{u_j\}_j$, define for all $i\in\N$,
$$ B_i = \{j\in\N \text{ : } \displaystyle \sum_{l=1}^{i-1} s_l < u_j\leq \sum_{l=1}^{i}s_l\}$$
with the convention that $\sum_{l=1}^{0} s_l=0$.
The $B_i$'s can be seen as the {blocks} of a partition of $\N$. 
We now create the sequence $\{\Tilde{B}_i\}_{i}$ through a reordering of those blocks by increasing order of their smallest element
with the convention that an empty block has a representative value of $\infty$. 

Now, if $\vect{s}\in\Delta\backslash\Delta^*$, we write $B_{\infty}=\{j\in\N\text{ : } \sum_{l=1}^{\infty} s_l < u_j\leq 1\}$ and apply the same procedure assuming that every element in $B_{\infty}$ forms its own block as a singleton.

From any partition of the integers $\{\Tilde{B}_i\}_i$ built from a pair of vectors $(\vect{s},\vect{u})$, we define the deterministic mapping
\begin{equation}\label{eq:LXi}
\ell^\Xi(j)=\{i\in\N:j\in \tilde{B}_i\}
\end{equation}
that associates at each integer the index of its block after reordering.
This mapping fully depends on $\vect{s}$ and $\vect{u}$ and will be made random and label-forgetful in the next section.

\begin{figure}
     \centering
     \includegraphics[width=0.5\linewidth]{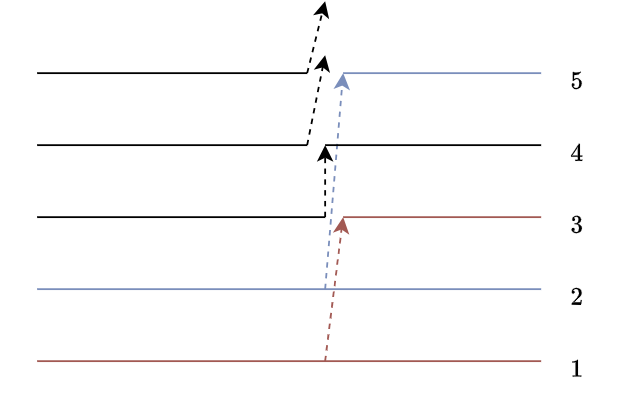}
     \caption{Illustration of the procedure when the blocks are $B_1=\{2,5,\dots\}, B_2=\{4,\dots\}$ and $B_3=\{1,3,\dots\}$. Their reordering gives $\Tilde{B}_1=B_3$, $\Tilde{B}_2=B_1$ and $\Tilde{B}_3=B_2$. 
     The particle at level 1 places a child at level 3. The particle at level 2 remains at the same level and places a child at level 5. The particle at level 3 is pushed to level 4. More generally, the particle at level $i$ is pushed at level $\inf \Tilde{B}_i$, i.e. the lowest available level, and its children are placed at the levels given by  $\Tilde{B}_i\backslash\{\inf \Tilde{B}_i\}$.}
     \label{graph_event}
 \end{figure}

One can interpret this mapping as a reproduction event. All the individuals of block $\tilde{B}_i$ are related to a mother particle at level $i$.
This particle is reallocated at level $\inf \Tilde{B}_i$ and its offspring, when they exist, are placed at levels given by $\Tilde{B}_i\backslash\{\inf \Tilde{B}_i\}$.
Figure \ref{graph_event} provides a graphical representation of such an event.
This construction ensures that the mapping satisfies a lookdown property: if $i_1<i_2$, then $\inf \Tilde{B}_{i_1}<\inf \Tilde{B}_{i_2}$.
Moreover, the label of a mother is always less than or equal to the labels of all of its offspring. 
This property will be the main ingredient for the construction of a discrete lookdown model associated to random atoms $\vect{s}$ and $\vect{u}$. 

One could think of taking for $\ell^\Xi(j)$ the infimum of the block containing $j$ instead of the index of the block, but that would break the property that the ranking between particles is maintained, described above. 
Let us stress nonetheless that the two approaches coincide when there is only one block, i.e. $s_1>0$ and $s_2=0$.
This case is related to the mapping defined in Example \ref{ex:coal}.

\subsection{The Poisson approximation technique for $\Xi$-lookdown construction}
\label{discrete}

Define a Poisson random measure $M$ on $\R_+\times\Delta\times[0,1]^{\N}$ with intensity $dt\otimes \nu(d\vect{s})\otimes\mathcal{U} (d\vect{u})$, where $\mathcal{U} (d\vect{u})=\bigotimes_{l=1}^{\infty} du_l$ and $\nu$ is a $\sigma$-finite measure on $\Delta$ such that 
\begin{equation}
\label{eqn:nu}
   \nu(\{s_1=0\})=0 \text{ and } \displaystyle\int_{\Delta} \Big(\sum s_l^2 \Big) \nu(d\vect{s})<\infty.
\end{equation}
This Poisson random measure will be the cornerstone to build both discrete and continuous dual forward and backward processes. 
Condition \eqref{eqn:nu} ensures that the rate of events affecting any fixed group of particles remains finite, and that points leading to null events are not considered. 
The finite measure $\Xi(d\vect{s})=(\sum s_l^2) \nu(d\vect{s})$ is the characteristic object of the processes in the sequel.

Fix a time horizon $T>0$. All processes will be constructed on the time interval $[0,T]$ and can be extended to $\R_+$ through standard arguments. 
Following the procedure in \cite{Ber06}, we can construct an exchangeable coalescent with characteristic measure $\Xi$ (or ${\nu}$) from the Poisson random measure $M$. 
Each atom $(t,\vect{s},\vect{u})$ of $M$ gives a coalescence time and a rule obtained from a paintbox event associated to the mass partition $\vect{s}$ and the assignment of the blocks $\vect{u}$ at time $t$.
Denote by $\{A^\Xi_t\}_t$ the resulting block-counting process of the coalescent, namely the Markov process with values in $\N$ that jumps from $n$ to $m$ at rate $q_{n,m}$ defined in \cite{MG21}, Proposition 2.1.
We know from \cite{GCS} that the latter is moment dual to a pure jump process $\{W^\Xi_t\}_{t}$ with values in $[0,1]$ given by
\begin{equation}\label{eq:diffXi}
 dW^{\Xi}_t = \int_{\Delta\times[0,1]^{\N}}  \sum_{l\ge1} s_l \left(\mathbb{1}_{\{u_l \leq W^{\Xi}_{t^-}\}} - W^{\Xi}_{t^-}\right) M(dt, d\vect{s}, d\vect{u}).
 \end{equation}
Its generator can be found in \cite{Five09}, Eq. (5.6).

  We aim to create discrete versions of these processes in order to apply the machinery developed in section \ref{partI}. 
To do so, fix $N\in\N$ and assume without loss of generality that $T={K}/{N}$, $K\in\N$. We can divide $[0,T]$ into $K$ sub-intervals of size ${1}/{N}$, and isolate (at most) one particular point of the Poisson random measure $M$ in each of these time intervals  (see Figure \ref{poisson_approx}). 
Intuitively, we take the "biggest" event in each sub-interval, according to some chosen rule. 
Rigorously, define the order relation on $[0,T]\times\Delta\times[0,1]^{\N}$ by $(t,\vect{s},\vect{u})\preceq (t',\vect{s'},\vect{u'})$ if and only if $ s_1< s'_1$ or $s_1= s'_1$ and $t\leq t'$. 
For $k\in[K]$, we consider $(t^{(k)},\vect{s}^{(k)}, \vect{u}^{(k)})$, the atom in the $k$-th  time interval chosen such that 
\begin{equation}
    \label{sk}
    (t^{(k)},\vect{s}^{(k)},\vect{u}^{(k)}) = \max\{(t,\vect{s},\vect{u})\text{ : } (t,\vect{s},\vect{u})\in M\cap(\left[\frac{{ k-1}}{N},\frac{{ k}}{N}\right)\times\Delta\times[0,1]^{\N})\}, 
\end{equation} 
where the maximum is taken according to the order $\preceq$. 
By convention, we will choose $\vect{s}=0$ if the set is empty.
It is easy to derive from \eqref{eqn:nu} that this maximum exists with probability $1$, as it ensures that the rates prevent an accumulation of points $s_1$ away from $0$. 
We then write $L^{\Xi,N}_{k}$ for the random mapping defined in \eqref{eq:LXi} and obtained from $\vect{s}^{(k)}$ and $\vect{u}^{(k)}$ in \eqref{sk}. 
It is straightforward to see the following.
\begin{lemma}\label{lem:LXi}
  The random mappings $\{L^{\Xi,N}_{k}\}_k$ are independent  and identically distributed  label-forgetful mappings that satisfy a lookdown principle.
\end{lemma}

Observe that this example provides a case of non-exchangeable label-forgetful mappings.
Let us consider the discrete-time block-counting process $\{A^{\Xi,N}_k\}_k$ defined by $A^{\Xi,N}_0=n\in \N$ and $A^{\Xi,N}_k{=|L^{\Xi,N}_{K-k+1}\circ\cdots\circ L^{\Xi,N}_{K}([n])|}$.
Letting $N$ grow reduces the size of the time intervals, and there exists therefore $N_0$ such that if $ N>N_0$, then $A^{\Xi,N}_{\lfloor Nt\rfloor}=A^{\Xi}_t$ for any $t$ such that $Nt\in[K]$.
In particular $ A_{\lfloor Nt\rfloor}^{\Xi,N}$ converges in distribution towards $A^{\Xi}_t$
 for any $t\in[0,T]$, fulfilling assumption $i)$ in Proposition~\ref{theo_cv}. 
 The reduction in size of the time intervals is illustrated in Figure~\ref{poisson_approx}.
 With the same argument, we can construct the random variables $L^\Xi_t(i)=\lim_{N\to\infty}L^{\Xi,N}_{K-\lfloor Nt\rfloor+1}\circ\cdots\circ L^{\Xi,N}_{K}(i)$ for any $i\in\N$ and $t\in[0,T]$ and the lookdown principle derives from the construction.  

\begin{figure}
    \centering
    \includegraphics[width=0.5\linewidth]{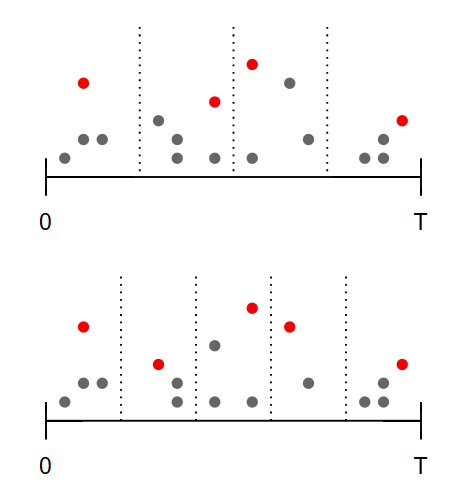}
    \caption{Representation
    of the realization of the Poisson random measure $M$ on $[0,T]\times\Delta\times[0,1]^{\N}$, where, on the y-axis, we only consider the first element $s_1$ of any point $\vect{s}\in\Delta$. Points in red are the selected points in each of the time intervals. Points are accumulating around 0 and become scarcer for larger values.}
    \label{poisson_approx}
\end{figure}
 
Label-forgetfulness entails the existence of a forward process $\{X^{\Xi,N}_k\}_k$ with values in $\{0,1\}^\N$ as in (\ref{X_forward}). 
Thus, we can define the associated de Finetti process $\{\Theta^{\Xi,N}_k\}_{ k}$. 
Their initial conditions lie in the exchangeability of $\{X^{\Xi,N}_0(i)\}_i$, and the value of $\Theta^{\Xi,N}_0=\lim_{p}\frac1p\sum_{i=1}^pX_0^{\Xi,N}(i)$. 
Repeating this construction for any $N\in\N$, we build a sequence of $[0,1]$-valued Markov processes $\{\Theta^{\Xi,N}\}_{N}$. 
Now, using the discrete duality resulting from Theorem \ref{thduality} and observing that conditions of Proposition \ref{theo_cv} are fulfilled, we are able to show that the sequence $\{\Theta^{\Xi,N}\}_{N}$ approximates the forward process dual to the block-counting of the $\Xi$-coalescent.  

\begin{Lemma}
\label{cor:cvxi}
    For each $N\in\N$ and $t\in[0,T]$, define the continuous-time process $\{{W}^{\Xi,N}_t\}_t$ by ${W}^{\Xi,N}_t=\Theta^{\Xi,N}_{\lfloor Nt \rfloor}$. 
    The sequence $\{{W}^{\Xi,N}\}_{N}$ weakly converges to $W^\Xi$ in $\D([0,T],[0,1])$ as $N\to\infty$.
\end{Lemma}

So far, we managed to obtain a countable representation in discrete time of the forward-in-time dual of the $\Xi$-coalescent block-counting process. As in \cite{Seven05,Five09,DK99}, particles are ranked according to the remaining time of their descendants in the system but this ranking does not break exchangeability. This is made clear thanks to the label-forgetfulness property.
We add in the next section the missing components to build a process whose forward in time type diffusion matches that of Theorem \ref{lazy}.

\subsection{The Poisson approximation technique for branching and pairwise coalescence} \label{sec:Xi+K}

The method we designed in the previous section consists in two steps: defining a label-forgetful mapping from atoms of a Poisson random measure and isolating those atoms thanks to a discretization of the time. It can be generalized to involve other phenomena. In particular, it can encompass the full family of exchangeable coalescents by completing the $\Xi$-events with Kingman binary mergers, and most importantly it can account for branching events.

The new branching events are triggered by elements $\vect{j}=(j_1,j_2)$ of $\N^2$ where $j_1$ provides the level of the branching particle and $j_2$ the number of added potential ancestors this particle takes. 
Precisely, the point $\vect{j}\in \N^2$ induces the mapping $\ell^B$ such that $\ell^B(j_1)=\{j_1,j_1+1,...,j_1+j_2\}$, $\ell^B(i)=\{i\}$ for $i<j_1$ and  $\ell^B(i)=\{i+j_2\}$ otherwise. 
We will incorporate those atoms in the Poisson point process, according to a measure whose restriction to $[0,T]\times\N^2$ has intensity
 $dt\otimes c\sum_{j_1\geq1}\delta_{j_1}\otimes\vect{p}$, where
 $\vect{p}$ is a probability measure with finite expectation and probability generating function $\Phi$.

Similarly, the new binary merger events will be induced by atoms in $C$ defined in Theorem~\ref{lazy} providing the level of the two particles that coalesce. 
The point $\vect{i}=(i_1,i_2)$ induces the mapping $\ell^K$ 
such that $\ell^K(i_1)=\ell^K(i_2)=i_1$. 
 For other integers, $\ell^K(i)=i$ if $i<i_2$ and $\ell^K(i)=i-1$ if $i>i_2$. 
The associated measure, restricted to $[0,T]\times C$, is $dt\otimes\sum_{i_1<i_2}\delta_{(i_1,i_2)}$. 
It is clear that both $\ell^K$ and $\ell^B$ respect the lookdown principle.

In the resulting (random) continuous-time branching-coalescing process, the former events lead to each level branching independently from each other at rate $c$, and the latter to each pair coalescing at rate $1$. The associated block-counting process involving these two kinds of events is dual to the strong solution of
$$ dW^{K+B}_t = \sqrt{W^{K+B}_t(1-W^{K+B}_t)}dB_t +  f(W^{K+B}_t)W^{K+B}_t(1-W^{K+B}_t)dt,$$
where $\{B_t\}_t$ stands for the standard Brownian motion and $f(x)=c(1-\Phi(x))/(1-x)$. 
The first term above arises from the Wright-Fisher component, dual to Kingman's coalescent block-counting process, and the second to the drift corresponding to frequency-dependent selection \cite{GCS}.

Once again, we wish to provide a discrete construction using these new events to apply the machinery developed in section~\ref{partI}.
To obtain a wider version of Lemma \ref{cor:cvxi} with selection drift and Wright-Fisher diffusion components, consider the Poisson random measure
$\Pi$ on $E:=\R_+\times\Delta\times[0,1]^{\N}\times C\times\N^2$ with intensity $dt\otimes\mu(d\vect{s},d\vect{u}, d\vect{i},d\vect{j})$,
with $\mu=\mu_\Xi+\mu_K+\mu_B$
and
$$\mu_\Xi=\nu\otimes \mathcal{U}\otimes\delta_0\otimes\delta_0,\quad \mu_K=\delta_0\otimes\delta_0\otimes\sum_{\vect i}\delta_{\vect{i}}\otimes\delta_0,\quad \mu_B=\delta_0\otimes\delta_0\otimes\delta_0\otimes\left(\sum_{j_1}\delta_{j_1}\otimes\vect {p}\right).$$
The measures $\delta_0$ have to be understood as Dirac masses on the null element of their respective spaces. 
The random measure $\Pi$ yields three different kinds of atoms.
Atoms $(t,\vect{s},\vect{u},0,0)$ lead to $\Xi$-events  determined by mappings $\ell^\Xi$ defined in \eqref{eq:LXi}.
Atoms $(t,0,0,\vect{i},0)$ lead to binary events associated to the pair $\vect{i}=(i_1,i_2)$ and the mapping $\ell^K$. 
Atoms $(t,0,0,0,\vect{j})$ lead to branching events associated to the pair $\vect{j}=(j_1,j_2)$ and the mapping $\ell^B$.

We can provide a similar discretization procedure as in the previous section.
Define the function $g$ on $E$ such that $$ g(t,\vect{s},\vect{u},\vect{i},\vect{j}):=s_1\mathbb{1}_{\vect{s}\neq0}+\frac{1}{i_1}\mathbb{1}_{\vect{i}\neq0}+\frac{1}{j_1}\mathbb{1}_{\vect{j}\neq0}$$ and the order relation $\preceq_E $ on $E$ by $(t,\vect{s},\vect{u},\vect{i},\vect{j})\preceq_E (t',\vect{s'},\vect{u'},\vect{i'},\vect{j'})$ if and only if $g(t,\vect{s},\vect{u},\vect{i},\vect{j})<g(t',\vect{s'},\vect{u'},\vect{i'},\vect{j'})$ or $t\leq t'$ in case of an equality. 
At $k$-th time interval of size $1/N$, we select the atom $(t^{(k)},\vect{s}^{(k)},\vect{u}^{(k)},\vect{i}^{(k)},\vect{j}^{(k)})$ which equals
$$ \max\{(t,\vect{s},\vect{u},\vect{i},\vect{j}): (t,\vect{s},\vect{u},\vect{i},\vect{j})\in \Pi\cap\left[\frac{{ k-1}}{N},\frac{{ k}}{N}\right)\times\Delta\times[0,1]^{\N}\times C\times\N^2\}$$
where the maximum is taken according to $\preceq_E$.

The proof of Theorem \ref{lazy} is finalized by adapting the same token as for Lemmas \ref{lem:LXi} and \ref{cor:cvxi} to the framework of the present section.
For each $N$, we can define a sequence of independent and identically distributed label-forgetful mappings $\{L^N_k\}_k$ from which we can build a discrete-time block-counting process $\{A^N_k\}_k$ that, once rescaled in time, converges towards the continuous-time block-counting process $\{A_t\}_t$ with dynamics in \eqref{eq:dynA}. 
We then obtain by duality the associated discrete-time de Finetti process $\{\Theta^{N}_k\}_k$ that, once rescaled in time, converges towards the unique solution to \eqref{eq:EDSW} thanks to Proposition \ref{theo_cv}.

Note that, once again, the rule to select the atom in each time window is arbitrary but it allows, as the time grid becomes smaller, to take every relevant atom  into account.
This provides a graphical construction of the lookdown model with selection studied in \cite{BP15} and the resulting duality relation extends that between the Wright-Fisher diffusion with selection and the ancestral selection  graph \cite{KN97}. The latter is precisely recovered by taking a measure $\vect{p}$ such that $p_2=1$, and replacing $\nu$ with $\delta_0$.

\color{black}
Using the same technique, it is possible to define even more general processes, and in particular involving broader branching events similar to that of Example \ref{ex:branching}. However, the definition of the associated Poisson random measure may become tedious, and we therefore present a more elegant method based on the observation in Example $\ref{ex:mix}$.

\subsection{The mixture approximation technique for lookdown with selection}\label{sec:mixture}
\label{selection}

When continuous-time block-counting processes involve several distinct components, such as branching-coalescing processes arising in Theorem \ref{lazy}, an alternative to the Poisson random measure approach of sections \ref{discrete} and \ref{sec:Xi+K} is possible. 
Rather than integrating each kind of event into a single large Poisson point process, one can define a random mapping that {mixes} events directly, with coin tosses determining which rule to apply at each step. 
This approach may be seen as a mixture of label-forgetful mappings. 
\begin{proposition}\label{propmix}
   For $i=1,2$, consider two sequences of label-forgetful mappings $\{L^{i.N}\}_N $ from which one can approximate forward and backward continuous-time processes $\{W^{i}_t\}_t$ and $\{A^{i}_t\}_t$ using Proposition \ref{theo_cv}. 
    Also suppose that $\{W^i_t\}_t$ has generator $\mathcal{F}^i$ and $\{A^i_t\}_t$ has generator $\mathcal{G}^i$.
    Then, if $L^N=YL^{1.N}+(1-Y)L^{2.N}$ and $Y$ is a Bernoulli random variable of parameter $p$ independent of both mappings,
    one can also construct a forward and a backward discrete-time process from the sequence of label-forgetful mappings $\{L^N\}_N$, whose discrete-time generators converge towards $p\mathcal{F}^1+(1-p)\mathcal{F}^2$ on the forward side, and $p\mathcal{G}^1+(1-p)\mathcal{G}^2$ on the backward side.
\end{proposition}
\begin{proof}
First recall from Example \ref{ex:mix} that ${L}^{N}$ is label-forgetful and thus define $\{{A}^{N}_k\}_k$ its associated backward-in-time block-counting process (starting at $n$ under $\P_n$), 
and $\{{\Theta}^{N}_k\}_k$ its forward-in-time de Finetti process. 
If $\{{A}^{1,N}_k\}_k$ and $\{{A}^{2,N}_k\}_k$ denote the block-counting processes associated to $L^{1,N}$ and $L^{2,N}$,
we have that 
\[
\begin{aligned}
\mathbb{E}_n[f({A}^{N}_1)] &= \mathbb{E}_n[f\left({Y}A^{1,N}_1+(1-{Y}){A}^{2,N}_1\right)]  \\ & = \mathbb{E}_n[{Y}f(A^{1,N}_1)+(1-{Y})f({A}^{2,N}_1)]\\ &  =p\mathbb{E}_n[f(A^{1,N}_1)]+(1-p)\mathbb{E}_n[f({A}^{2,N}_1)]    
\end{aligned}
\]
for measurable functions $f:\N\to \R$.  
   Through the combination of Theorems 17.25 and 17.28 in \cite{Kal92}, we get that, for any $t>0$, ${A}^{N}_{\lfloor a_Nt \rfloor}$ converges towards the marginal ${A}_t$ of a $\N$-valued Markov process with generator $p\mathcal{G}^{1}+(1-p){\mathcal{G}}^{2}$. 

Taking $f(y)=x^y$ above for some $x\in(0,1)$, we can apply Theorem \ref{thduality} and get, for any $n$,
$$\mathbb{E}_x\left[\left(\Theta_1^N\right)^n\right]= p\mathbb{E}_x\left[\left(\Theta_1^{1,N}\right)^n\right]+ (1-p) \mathbb{E}_x\left[\left(\Theta_1^{2,N}\right)^n\right].
$$
 By observing that the class of monomials is dense in the set of bounded measurable functions from $[0,1]$ to $\R$, we obtain that the discrete generator of $\{\Theta^N_k\}_k$ also converges towards $p\mathcal{F}^1+(1-p)\mathcal{F}^2$.
\end{proof}

One can create any mixed process with the right ingredients as long as they all involve label-forgetful choices of ancestors and harbor block-counting processes converging to the adequate known continuous-time Markov processes, see Figure \ref{fig:mixed_event}. 
To illustrate this technique, we show how one can easily obtain processes from Theorem \ref{lazy}, with even more general selection associated to branching events where, even in the continuous-time limit, all particles branch simultaneously. This model is a particular case of the Wright-Fisher diffusion with selection in random environment studied in (\cite{GSW25}, Eq. (4)). 

\begin{figure}
    \centering
    \includegraphics[width=0.99\linewidth]{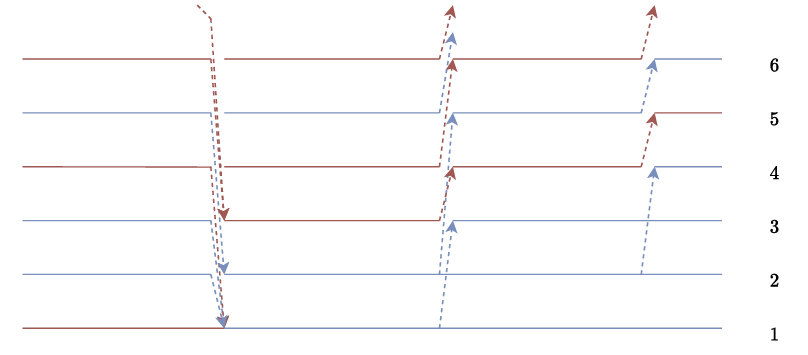}
    \caption{Illustration of mappings used in the proof of Theorem \ref{theomix}. Colors indicate different types. We observe three events, the first one is a {frequency-dependent selection} event, in which $L(\{1\})=\{1,2,3,4\}$ and its type changes from red to blue. Next we observe a $\Xi$-event and finally a Kingman-like event.
    }
    \label{fig:mixed_event}
\end{figure}

\begin{theorem}\label{theomix}
     Recall the notations of Theorem \ref{lazy}.
     Additionally, consider a probability measure on $\N\backslash\{1\}$ with probability generating function $\varphi$ and mean $m$, and a Poisson process $\{N_t\}_t$ with intensity $d/(m-1)$ for $d>0$.
     There exists a sequence of label-forgetful mappings $\{\hat{L}^{N}\}_N$  satisfying a lookdown principle and from which one can construct 
 a $\{0,1\}^\N$-valued process  whose induced frequency process $\{\hat W_t\}_{t}$ is the solution of the stochastic differential equation
\begin{align}\label{eq:EDShW}
           d\hat{W}_t & = \sqrt{\hat{W}_t(1-\hat{W}_t)}dB_t + f(\hat{W}_t)\hat W_t(1-\hat{W}_t)dt \notag \\ &  \qquad +(\varphi(\hat{W}_{t-})-\hat{W}_{t-})dN_t\notag \\ & \qquad + \int_{\Delta\times[0,1]^{\N}}  \sum_{l\ge1} s_l \left(\mathbb{1}_{\{u_l \leq \hat{W}_{t^-}\}} -\hat{W}_{t^-}\right) {M}(dt, d\vect{s}, du),
        \end{align}
      and its dual branching-coalescing process is the $L$ induced process $\{\hat A_t\}_t$.

\end{theorem}

\begin{proof}
We apply Proposition \ref{propmix} with four mappings instead of two: $L^{\Xi,N}$ from section \ref{discrete}, and three new mappings $L^{K,N}$, $\tilde{L}^{B,N}$ and $\hat{L}^{B,N}$ involving binary coalescing and branching events that we describe in the sequel.

For the binary coalescing part, adapt the construction in section \ref{xi_event} by considering $N$ i.i.d. uniform random variables on $[0.1]$, $U_1,\dots, U_N$.
Define for all $i\le N$,
$$ B_i = \{j\le N \text{ : } \displaystyle \frac{i-1}N < U_j\leq \frac{i}N\}.$$
Remove the empty blocks and define $\{\Tilde{B}_i\}_{i}$ through a reordering by increasing order of their smallest element.
Now, set for $j\le N$
$$L^{k,N}(j)=\{i\in\N:j\in \tilde{B}_i\}$$
and $L^{k,N}(j)=j$ for $j>N$.
Let us define $\{ A^{K,N}_k\}_k$ the block-counting chain obtained  by iterating $ L^{K,N}$.  
It is clear that this behaves like that of a Wright-Fisher model of size $N$.
This implies that the rescaled process converges towards Kingman's coalescent block-counting process, dual to the classical Wright-Fisher diffusion.

For the first branching part, define $\tilde{L}^{B,N}$ as in Example \ref{ex:branching} through i.i.d. random variables $\{\xi_i\}_{i}$ such that $\P(\xi_1=j)=c p_j/N$, $j\ge2$, and $\P(\xi_1=1)=1-c/N$ for $c>0$ and a positive sequence $\{p_j\}_j$ such that $\sum_{j\ge2}p_j=1$.
Such a mapping can be described sequentially. 
Starting from level 1, each particle draws $j$ arrows with probability $c p_j/N$ (\emph{resp.} one arrow with probability $1-c/N$) towards the $j$ next available particles (\emph{resp.} the next available particle). 
This label-forgetful satisfies a lookdown property.
 Moreover, a discrete-time block-counting process $\{\tilde{A}^{B,N}_k\}_k$ can be obtained by iterating this procedure.
Under $\P_n$, the law of the block-counting process starting at $n$, and rescaling time,
it is clear, due to the independence of the random variables $\{\xi_i\}_i$, that the sequence of processes $\{\{\tilde{A}^{B,N}_{\lfloor Nt\rfloor}\}_t\}_N$ converges weakly,  as $N\to\infty$, to a continuous-time Galton-Watson process started at $n$,  where each individual branches independently from the others at rate $c$ (the progeny being governed by $\{p_j\}_j$).

For the second branching component,
    define $\hat{L}^{B,N}$ as follows.
     With probability $d/N$, $\hat L^{B,N}$ is the mapping from Example~\ref{ex:branching} built through i.i.d. random variables $\{\xi_i\}_{i}$ with      probability generating function $\varphi$ and mean $m$.
     With probability $1-d/N$, $\hat L^{B,N}$ is the identity, that is $\hat L^{B,N}(i)=i$. 
     Again, let us define $\{\hat A^{B,N}_k\}_k$ the block-counting chain obtained  by iterating $\hat L^{B,N}$.  
 By \cite{GSW25}, the rescaled process ${\hat A^{B,N}}$ converges towards the continuous-time block-counting process dual to the solution of 
$$ d\hat W^{B}_t = (\varphi(\hat W^{B}_{t-})-\hat W^{B}_{t-})dN_t.$$

It just remains to plug these pieces into the mixture technique developed in Proposition \ref{propmix} and to identify the limit generators with those of the limiting processes $\{\hat{A}_t\}_t$ and $\{\hat{W}_t\}_t$.
\end{proof}

\section*{Acknowledgments}
AO acknowledges partial support from the chaire program "Mathematical modeling and biodiversity" (Ecole Polytechnique, Museum National d’Histoire Naturelle, Veolia Environnement, Fondation X). AO is supported by a grant from Fondation CFM pour la recherche which had no role in the studies or in the decision to publish.

ASJ acknowledges Amandine Véber and MAP5 for their hospitality, and DGAPA-PAPIIT-UNAM grant IN-105726 and FSMP for their partial support. 

The authors thank Jim Pitman for insightful discussions.

\end{document}